\documentclass[10pt]{article}%

\usepackage[utf8]{inputenc}

\usepackage{amsmath}

\usepackage{amsfonts}
\usepackage{mathrsfs}
\usepackage{amssymb, color}
\usepackage[linkcolor=black,anchorcolor=black,citecolor=black]{hyperref}
\usepackage{graphicx}
\numberwithin{equation}{section}
\usepackage[body={15.5cm,21cm}, top=3cm]{geometry}%
\providecommand{\U}[1]{\protect\rule{.1in}{.1in}}
\providecommand{\U}[1]{\protect \rule{.1in}{.1in}}
\newtheorem{theorem}{Theorem}[section]

\newtheorem{definition}[theorem]{Definition}

\newtheorem{lemma}[theorem]{Lemma}

\newtheorem{proposition}[theorem]{Proposition}
\newtheorem{remark}[theorem]{Remark}

\newtheorem{assumption}[theorem]{Assumption}
\newenvironment{proof}[1][Proof]{\noindent \textbf{#1.} }{\  \rule{0.5em}{0.5em}}

\def \E{\mathsf{E}}

\def \P{\mathsf{P}}

\usepackage [numbers,sort&compress] {natbib}

\begin{document}
	\title{Backward Stochastic Differential Equations with Double Mean Reflections}
	\author{ 	Hanwu Li\thanks{Research Center for Mathematics and Interdisciplinary Sciences, Shandong University, Qingdao 266237, Shandong, China. lihanwu@sdu.edu.cn.}
	\thanks{Frontiers Science Center for Nonlinear Expectations (Ministry of Education), Shandong University, Qingdao 266237, Shandong, China.}}
	\date{}
	\maketitle
	
	\begin{abstract}
        In this paper, we study the backward stochastic differential equation (BSDE) with two nonlinear mean reflections, which means that the constraints are imposed on the distribution of the solution but not on its paths. Based on the backward Skorokhod problem with nonlinear constraints, we obtain the existence and uniqueness result by constructing a contraction mapping. When the constraints are linear, the solution can be approximated by a family of penalized mean-field BSDEs. 
	\end{abstract}


    \textbf{MSC-classification}: 60H10
	
\section{Introduction}

In 1997, El Karoui et al. \cite{EKPPQ} first considered the reflected backward stochastic differential equations (RBSDEs) of the following form
\begin{align}\label{intro1}
\begin{cases}
Y_t=\xi+\int_t^T f(s,Y_s,Z_s)ds-\int_t^T Z_s dB_s+(K_T-K_t),\\
Y_t\geq S_t, \ t\in[0,T],\\
K_0=0, \ K \textrm{ is a nondecreasing process such that } \int_0^T (Y_s-S_s)dK_s=0.
\end{cases}
\end{align}
Compared with the classical BSDE, the first component of the solution is required to be above a given process $S$. To this end, an additional nondecreasing process $K$ should be added in this equation to push the solution upwards. It is natural to assume that such process behaves in a minimal way, i.e., it only increases when the solution hits the boundary $S$. Mathematically, this condition is written as the Skorokhod condition. Then, Cvitani\'{c} and Karatzas \cite{CK} investigated the BSDEs with double reflection, where the solution lies between two prescribed obstacles. It was shown that the solution to a reflected BSDE and to a doubly reflected BSDE corresponds to the value function of an optimal stopping problem and a Dynkin game, respectively. Due to its wide applications in finance, economics and partial differential equations (see, e.g., \cite{BCFE,EPQ1,EPQ2,WY}), the theory of reflected BSDEs has attracted a great deal of attention to generalize results to the case when the obstacles are not continuous or to the case when the coefficient does not satisfy Lipschitz condition. The readers may refer to the papers \cite{CM,DQS,GIOOQ,GIOQ,HLM,K1,K2,KLQT,PX} and the references therein.

It is worth pointing out that in all the above mentioned papers, the constraints depends on the path of the solution. Recently, Briand, Elie and Hu \cite{BEH} proposed the BSDE with mean reflection, where the constraint is given in terms of the distribution of the solution. More precisely, the second condition in \eqref{intro1} is written as 
\begin{align*}
\E[l(t,Y_t)]\geq 0, \ t\in[0,T],
\end{align*}
where $l$ is some given loss function. In this situation, the nondecreasing function $K$ is required to be a deterministic function satisfying the Skorokhod condition. This type of BSDEs can be applied to the superhedging of claims under running risk management constraint. Then, Falkowski and S\l omi\'{n}ski \cite{FS} considered the case of mean reflection with two constraints, that is, 
\begin{align*}
\E[l(t,Y_t)]\in [l_t,r_t], \ t\in[0,T],
\end{align*} 
and $K$ is a deterministic function of bounded variation satisfying 
\begin{align*}
\int_0^T( \E[l(s,Y_s)]-l_s )dK^+_s=0, \ \int_0^T( \E[l(s,Y_s)]-r_s )dK^-_s=0,
\end{align*}
where $K^+,K^-$ represent the positive and negative part of the Jordan decomposition of the function $K$. They also established the connection between the expectation of the solutions and the value functions of some appropriate deterministic optimization problems. The readers may refer to \cite{BH,CHM,DEH,HHLLW} for some related topic concerning mean reflected BSDEs.  

In the present paper, we consider the BSDEs with two nonlinear reflecting boundaries. That is, given two nonlinear loss functions $L,R$ with $L\leq R$, the first component $Y$ of the solution should satisfy the following conditions
\begin{align*}
\E[L(t,Y_t)]\leq 0\leq \E[R(t,Y_t)], \ t\in[0,T]
\end{align*} 
and the minimality conditions for the bounded variation function $K$ are defined accordingly. The deterministic compensator $K$ in \cite{FS} is constructed based on an appropriately defined backward Skorokhod problem. Motivated by this approach, we first study the backward Skorokhod problem with two nonlinear reflecting boundaries based on the results in \cite{Li}, which is a building block for the bounded variation function $K$ in our framework. Then, we present the existence and uniqueness result by constructing contraction mapping. The second effective method is approximation via penalization, which is frequently used for reflected BSDEs (see, e.g., \cite{CK,EKPPQ}). However, when applying this method for BSDEs with mean reflections, the family of penalized BSDEs turns into the family of penalized mean-field BSDEs (see \cite{BLP}), which makes it difficult to obtain some monotone convergence result. Nevertheless, when the reflecting boundaries are linear, i.e., considering the following constraints
\begin{align*}
\E[Y_t]\in[ l_t,r_t], \ t\in[0,T],
\end{align*}  
we are able to construct the solution as the limit of penalized BSDEs of the Mckean-Vlasov type. 

This paper is organized as follows. We first formulate the BSDE with double mean reflections in details in Section 2. In Section 3, we recall some basic results about deterministic Skorokhod problem and then extend it to the backward case. The existence and uniqueness result and some properties of BSDE with double mean reflections are studied in Section 4. In the last section, when the reflecting boundaries are in a linear fashion, we construct the solution by a penalization method.

\section{Problem formulation}

Throughout this paper, we are given a finite time horizon $T>0$ and a filtered probability space $(\Omega,\mathcal{F},\mathbb{F},\P)$ satisfying the usual conditions of right continuity and completeness. Let $B$ be a standard Brownian motion. The following notations are frequently used in this paper.

\begin{itemize}
\item $L^2(\mathcal{F}_t)$: the set of real-valued $\mathcal{F}_t$-measurable random variable $\xi$ such that $\E[|\xi|^2]<\infty$, $t\in[0,T]$;
\item $\mathcal{S}^2$: the set of real-valued adapted continuous processes $Y$ on $[0,T]$ with $\E[\sup_{t\in[0,T]}|Y_t|^2]<\infty$;
\item $\mathcal{H}^2$: the set of real-valued predictable processes $Z$ such that $\E[\int_0^T|Z_t|^2dt]<\infty$;
\item $C[0,T]$: the set of continuous functions from $[0,T]$ to $\mathbb{R}$;
\item $BV[0,T]$: the set of functions in $C[0,T]$ starting from the origin with bounded variation on $[0,T]$;
\item $I[0,T]$: the set of functions in $C[0,T]$ starting from the origin which is nondecreasing.
\end{itemize}

The main purpose of this paper is to study the BSDE with double mean reflections of the following type
\begin{equation}\label{nonlinearyz}
\begin{cases}
Y_t=\xi+\int_t^T f(s,Y_s,Z_s)ds-\int_t^T Z_s dB_s+K_T-K_t, \\
\E[L(t,Y_t)]\leq 0\leq \E[R(t,Y_t)], \\
K_t=K^R_t-K^L_t, \int_0^T \E[R(t,Y_t)]dK_t^R=\int_0^T \E[L(t,Y_t)]dK^L_t=0,
\end{cases}
\end{equation}
where $K^R,K^L\in I[0,T]$. Here, mean reflection means that the constraints do not depends on the path of the solution but on the distribution of the solution.

\begin{remark}
(i) If $L\equiv-\infty$, the BSDE with double mean reflection degenerates into the BSDE with mean reflection studied in \cite{BEH}.

\noindent (ii) Given $l,u\in C[0,T]$ with $l_t\leq u_t$, $t\in[0,T]$, let $L(t,x)=h(t,x)-u_t$ and $R(t,x)=h(t,x)-l_t$. The BSDE with double mean reflection turns into the BSDE with mean reflection and two constraints investigated in \cite{FS}.
\end{remark}  

The parameters of BSDE with double mean reflection consist of the terminal condition $\xi$, the driver (or coefficient $f$) and the loss functions $L,R$. We propose the following assumptions on these parameters.
\begin{assumption}\label{ass2}
The functions $L,R:\Omega\times [0,T]\times\mathbb{R}\rightarrow \mathbb{R}$ are measurable maps with respect to $\mathcal{F}_T\times \mathcal{B}([0,T])\times \mathcal{\mathbb{R}}$ satisfying the following conditions:
\begin{itemize}
\item[(1)] for any fixed $(\omega,x)\in \Omega\times\mathbb{R}$, $L(\omega,\cdot,x), R(\omega,\cdot,x)$ are continuous;
\item[(2)]  $\E[\sup_{t\in[0,T]}|L(t,0)|]<\infty$, $\E[\sup_{t\in[0,T]}|R(t,0)|]<\infty$;
\item[(3)] for any fixed $(\omega,t)\in \Omega\times [0,T]$, $L(\omega,t,\cdot),R(\omega,t,\cdot)$ are strictly increasing and there exists two constants $0<c<C$ such that for any $x,y\in \mathbb{R}$,
\begin{align*}
&c|x-y|\leq |L(\omega,t,x)-L(\omega,t,y)|\leq C|x-y|,\\
&c|x-y|\leq |R(\omega,t,x)-R(\omega,t,y)|\leq C|x-y|;
\end{align*}
\item[(4)] $\inf_{\omega,t,x} (R(\omega,t,x)-L(\omega,t,x))>0$.
\end{itemize}
\end{assumption}

\begin{assumption}\label{assf}
The driver $f$ is a map from $\Omega\times[0,T]\times \mathbb{R}\times\mathbb{R}$ to $\mathbb{R}$. For 
each fixed $(y,z)$, $f(\cdot,\cdot,y,z)$ is progressively measurable. There exists $\lambda>0$ such that for any $t\in[0,T]$ and any $y,y',z,z'\in\mathbb{R}$
\begin{align*}
|f(t,y,z)-f(t,y',z')|\leq \lambda(|y-y'|+|z-z'|)
\end{align*}
and 
\begin{align*}
\E[\int_0^T |f(t,0,0)|^2dt]<\infty.
\end{align*}
\end{assumption}

\begin{remark}
In \cite{CHM}, the authors considered the mean-field doubly reflected BSDE of the following type
\begin{displaymath}
\begin{cases}
Y_t=\xi+\int_t^T f(s,Y_s,\E[Y_s],Z_s)ds-\int_t^T Z_s dB_s+K^+_T-K^+_t-K^-_T+K^-_t, \\
h(t,\omega,Y_t,\E[Y_t])\leq Y_t\leq g(t,\omega,Y_t,\E[Y_t]), \\
\int_0^T (Y_s-h(s,Y_s,\E[Y_s]))dK_s^+=\int_0^T (Y_s-g(s,Y_s,\E[Y_s]))dK^-_s=0,
\end{cases}
\end{displaymath}
where $h,g$ are mappings from $[0,T]\times \Omega\times \mathbb{R}^2$ into $\mathbb{R}$ satisfying the following Lipschitz condition, i.e., there exist pairs of positive constants $(\gamma_1,\gamma_2)$, $(\beta_1,\beta_2)$ such that for any $t\in[0,T]$, $x,x',y,y'\in \mathbb{R}$,
\begin{align*}
&|h(t,x,y)-h(t,x',y')|\leq \gamma_1|x-y|+\gamma_2|x'-y'|,\\
&|g(t,x,y)-g(t,x',y')|\leq \beta_1|x-y|+\beta_2|x'-y'|.
\end{align*}
Under the following additional assumption on the Lipschitz constants, 
\begin{align*}
(\gamma_1+\gamma_2+\beta_1+\beta_2)^{1/2}[4(\gamma_1+\beta_1)+(\gamma_2+\beta_2)]^{1/2}<1,
\end{align*}
they show that the mean-field doubly reflected BSDE has a unique solution $(Y,Z,K^+,K^-)$. Compared with this result, we do not need the assumption on the Lipschitz constants to ensure the existence of solutions to  BSDEs with double mean reflections.
\end{remark}

\section{Backward Skorokhod problem}

\subsection{Skorokhod problem with two nonlinear reflecting boundaries}

In this subsection, we first recall some basic result of the Skorokhod problem with nonlinear reflecting boundaries in \cite{Li}. Let $D[0,\infty)$ be the set of real-valued right-continuous functions with left limits (or, c\`{a}dl\`{a}g functions). $I[0,\infty)$, $C[0,\infty)$, $BV[0,\infty)$ and $AC[0,\infty)$ denote the subset of $D[0,\infty)$ consisting of nondecreasing functions, continuous functions, functions of bounded variation and absolutely continuous functions, respectively.  For any $K\in BV[0,\infty)$ and $t\geq 0$, $|K|_t$ represents the total variation of $K$ on $[0,t]$. 

\begin{definition}\label{def'}
Let $s\in D[0,\infty)$, $g,h:[0,\infty)\times \mathbb{R}\rightarrow \mathbb{R}$ be two functions with $g\leq h$. A pair of functions $(x,k)\in D[0,\infty)\times BV[0,\infty)$ is called a solution of the Skorokhod problem for $s$ with nonlinear constraints $g,h$ ($(x,k)=\mathbb{SP}_g^h(S)$ for short) if 
\begin{itemize}
\item[(i)] $x_t=s_t+k_t$;
\item[(ii)] $g(t,x_t)\leq 0\leq h(t,x_t)$;
\item[(iii)] $k_{0-}=0$ and $k$ has the decomposition $k=k^h-k^g$, where $k^r,k^l$ are nondecreasing functions satisfying
\begin{align*}
\int_0^\infty I_{\{g(s,x_s)<0\}}dk^g_s=0, \  \int_0^\infty I_{\{h(s,x_s)>0\}}dk^h_s=0.
\end{align*}
\end{itemize}
\end{definition}

We propose the following assumption on the functions $g,h$.
\begin{assumption}\label{ass}
The functions $g,h:[0,\infty)\times \mathbb{R}\rightarrow \mathbb{R}$ satisfy the following conditions
\begin{itemize}
\item[(i)] For each fixed $x\in\mathbb{R}$, $g(\cdot,x),h(\cdot,x)\in D[0,\infty)$;
\item[(ii)] For any fixed $t\geq 0$, $g(t,\cdot)$, $h(t,\cdot)$ are strictly increasing;
\item[(iii)] There exists two positive constants $0<c<C<\infty$, such that for any $t\geq 0$ and $x,y\in \mathbb{R}$,
\begin{align*}
&c|x-y|\leq |g(t,x)-g(t,y)|\leq C|x-y|,\\
&c|x-y|\leq |h(t,x)-h(t,y)|\leq C|x-y|.
\end{align*} 
\item[(iv)] $\inf_{(t,x)\in[0,\infty)\times\mathbb{R}}(h(t,x)-g(t,x))>0$.
\end{itemize}
\end{assumption}

\begin{theorem}[\cite{Li}]\label{FSP}
Suppose that $g,h$ satisfy Assumption \ref{ass}. For any given $s\in D[0,\infty)$, there exists a unique pair of solution to the Skorokhod problem $\mathbb{SP}_g^h(s)$.
\end{theorem}

\begin{remark}
Given $s\in D[0,\infty)$, for any $t\geq 0$, let $\phi_t$, $\psi_t$ satisfy the following equations, respectively
\begin{align*}
g(t,s_t+\phi_t)=0, \ h(t,s_t+\psi_t)=0.
\end{align*}
Then, the second component of the solution to the Skorokhod problem $\mathbb{SP}_g^h(s)$ can be represented as follows:
\begin{align*}
k_t=\min\left([-\phi^-_0]\vee \sup_{r\in[0,t]}\psi_r,\inf_{s\in[0,t]}\left[\phi_s\vee \sup_{r\in[s,t]}\psi_r\right] \right).
\end{align*}
\end{remark}

The following proposition provides the continuity property of the solution to the Skorokhod problem with respect to input function and reflecting boundary functions.
\begin{proposition}[\cite{Li}]\label{prop4.1}
Suppose that $(g^i,h^i)$ satisfy Assumption \ref{ass}, $i=1,2$. Given $s^i\in D[0,\infty)$, let $(x^i,k^i)$ be the solution to the Skorokhod problem $\mathbb{SP}_{g^i}^{h^i}(s^i)$. Then, we have
\begin{align*}
\sup_{t\in[0,T]}|k^1_t-k^2_t|\leq \frac{C}{c}\sup_{t\in[0,T]}|s^1_t-s^2_t|+\frac{1}{c}
(\bar{g}_T\vee\bar{h}_T),
\end{align*}
where
\begin{align*}
&\bar{g}_T:=\sup_{(t,x)\in[0,T]\times\mathbb{R}}|g^1(t,x)-g^2(t,x)|,\\
&\bar{h}_T:= \sup_{(t,x)\in[0,T]\times\mathbb{R}}|h^1(t,x)-h^2(t,x)|.
\end{align*}
\end{proposition}

\begin{remark}
If $s\in C[0,\infty)$ and (i) of Assumption \ref{ass} is replaced by $g(\cdot,x),h(\cdot,x)\in C[0,\infty)$ for any $x\in\mathbb{R}$, then each component of solution to the Skorokhod problem $\mathbb{SP}_g^h(s)$ is continuous.
\end{remark}

\subsection{Backward Skorokhod problem with two nonlinear reflecting boundaries}

In this subsection, we consider the backward Skorokhod problem, which is the building block for BSDEs with double mean reflections \eqref{nonlinearyz}.
\begin{definition}\label{def}
Let $s\in C[0,T]$, $a\in \mathbb{R}$ and $l,r:[0,T]\times \mathbb{R}\rightarrow \mathbb{R}$ be two functions such that $l\leq r$ and $l(T,a)\leq 0\leq r(T,a)$. A pair of functions $(x,k)\in C[0,T]\times BV[0,T]$ is called a solution of the backward Skorokhod problem for $s$ with nonlinear constraints $l,r$ ($(x,k)=\mathbb{BSP}_l^r(s,a)$ for short) if 
\begin{itemize}
\item[(i)] $x_t=a+s_T-s_t+k_T-k_t$;
\item[(ii)] $l(t,x_t)\leq 0\leq r(t,x_t)$, $t\in[0,T]$;
\item[(iii)]  $k$ has the decomposition $k=k^r-k^l$, where $k^r,k^l\in I[0,T]$ satisfy 
\begin{align}\label{iii}
\int_0^T I_{\{l(s,x_s)<0\}}dk^l_s=0, \  \int_0^T I_{\{r(s,x_s)>0\}}dk^r_s=0.
\end{align}
\end{itemize}
\end{definition}

We first introduce the assumptions made for the reflecting boundary functions. Compared with the assumption made for the forward Skorokhod problem (see Assumption \ref{ass}), these two functions are defined on finite time interval and they are assumed to be continuous.
\begin{assumption}\label{ass1}
The functions $l,r:[0,T]\times \mathbb{R}\rightarrow \mathbb{R}$ satisfy the following conditions
\begin{itemize}
\item[(i)] For each fixed $x\in\mathbb{R}$, $l(\cdot,x),r(\cdot,x)\in C[0,T]$;
\item[(ii)] For any fixed $t\in[0,T]$, $l(t,\cdot)$, $r(t,\cdot)$ are strictly increasing;
\item[(iii)] There exists two positive constants $0<c<C<\infty$, such that for any $t\in [0,T]$ and $x,y\in \mathbb{R}$,
\begin{align*}
&c|x-y|\leq |l(t,x)-l(t,y)|\leq C|x-y|,\\
&c|x-y|\leq |r(t,x)-r(t,y)|\leq C|x-y|.
\end{align*} 
\item[(iv)] $\inf_{(t,x)\in[0,T]\times\mathbb{R}}(r(t,x)-l(t,x))>0$.
\end{itemize}
\end{assumption}

\begin{theorem}\label{BSP}
Let Assumption \ref{ass1} hold. For any given $s\in C[0,T]$ and $a\in \mathbb{R}$ with $l(T,a)\leq 0\leq r(T,a)$, there exists a unique solution to the backward Skorokhod problem $(x,k)=\mathbb{BSP}_l^r(s,a)$.
\end{theorem}

\begin{proof}
For any $t\in[0,T]$, set
\begin{align*}
\bar{s}_t:=a+s_T-s_{T-t}, \ \bar{l}(t,x):=l(T-t,x), \ \bar{r}(t,x):=r(T-t,x).
\end{align*}
We may check that $\bar{s}\in C[0,T]$ and $\bar{l},\bar{r}$ satisfy Assumption \ref{ass}. Besides,
\begin{align*}
\bar{l}(0,\bar{s}_0)=l(T,a)\leq  0\leq r(T,a)= \bar{r}(0,\bar{s}_0).
\end{align*}
By Theorem \ref{FSP}, there exists a unique solution to the Skorokhod problem $(\bar{x},\bar{k})=\mathbb{SP}_{\bar{l}}^{\bar{r}}(\bar{s})$. Clearly, we have $\bar{k}_0=0$. Let $\bar{k}^l$, $\bar{k}^r$ be two nondecreasing functions such that $\bar{k}=\bar{k}^r-\bar{k}^l$ and 
\begin{align*}
\int_0^T I_{\{\bar{l}(s,\bar{x}_s)<0\}}d\bar{k}^l_s=0, \  \int_0^T I_{\{\bar{r}(s,\bar{x}_s)>0\}}d\bar{k}^r_s=0.
\end{align*}
 For any $t\in[0,T]$, set 
\begin{align*}
x_t:=\bar{x}_{T-t}, \ k_t:=\bar{k}_T-\bar{k}_{T-t}.
\end{align*}
We claim that $(x,k)$ is the solution to the backward Skorokhod problem $\mathbb{BSP}_l^r(s,a)$. Indeed, for any $t\in[0,T]$, we have
\begin{align*}
&x_t=\bar{x}_{T-t}=\bar{s}_{T-t}+\bar{k}_{T-t}=a+s_T-s_t+k_T-k_t,\\
&l(t,x_t)=\bar{l}(T-t,\bar{x}_{T-t})\leq 0\leq \bar{r}(T-t,\bar{x}_{T-t})=r(t,x_t),\\
&\int_0^T I_{\{l(t,x_t)<0\}}dk^l_t=-\int_0^T I_{\{\bar{l}(T-t,\bar{x}_{T-t})<0\}}d\bar{k}^l_{T-t}=\int_0^T I_{\{\bar{l}(t,\bar{x}_t)<0\}}d\bar{k}^l_t=0,\\
&\int_0^T I_{\{r(t,x_t)>0\}}dk^r_t=-\int_0^T I_{\{\bar{r}(T-t,\bar{x}_{T-t})>0\}}d\bar{k}^r_{T-t}=\int_0^T I_{\{\bar{r}(t,\bar{x}_t)>0\}}d\bar{k}^r_t=0,
\end{align*}
where $k^l_t=\bar{k}^l_T-\bar{k}^l_{T-t}$, $k^r_t=\bar{k}^r_T-\bar{k}^r_{T-t}$. The proof is complete.
\end{proof}

Given $a^i\in\mathbb{R}$, $s^i\in C[0,T]$, $l^i,r^i$ satisfying Assumption \ref{ass1} and $l^i(T,a^i)\leq 0\leq r^i(T,a^i)$, $i=1,2$, 
 let $(x^i,k^i)$, $(\bar{x}^i,\bar{k}^i)$ be the solution to the backward Skorokhod problem $\mathbb{BSP}_{l^i}^{r^i}(s^i,a^i)$ and the solution to the forward Skorokhod problem $\mathbb{SP}_{\bar{l}^i}^{\bar{r}^i}(\bar{s}^i)$, $i=1,2$, where $\bar{l}^i(t,x)=l^i(T-t,x)$, $\bar{r}^i(t,x)=r^i(T-t,x)$, $\bar{s}^i_t=a^i+s^i_T-s^i_{T-t}$.  By the proof of Theorem \ref{BSP}, we have 
$$
x^i_t=\bar{x}^i_{T-t}, \ k^i_T-k^i_t=\bar{k}^i_{T-t}.
$$
Besides, by Proposition  \ref{prop4.1}, we have
\begin{align*}
\sup_{t\in[0,T]}|\bar{k}^1_t-\bar{k}^2_t|\leq &\frac{C}{c}|\bar{s}^1_t-\bar{s}^2_t|+\frac{1}{c}(\bar{L}_T\vee\bar{R}_T)\\
\leq & \frac{C}{c}|a^1-a^2|+2\frac{C}{c}\sup_{t\in[0,T]}|s^1_t-s^2_t|+\frac{1}{c}(\bar{L}_T\vee\bar{R}_T),
\end{align*}
where 
\begin{align*}
&\bar{L}_T=\sup_{(t,x)\in[0,T]\times \mathbb{R}}|\bar{l}^1(t,x)-\bar{l}^2(t,x)|=\sup_{(t,x)\in[0,T]\times \mathbb{R}}|{l}^1(t,x)-{l}^2(t,x)|,\\
&\bar{R}_T=\sup_{(t,x)\in[0,T]\times \mathbb{R}}|\bar{r}^1(t,x)-\bar{r}^2(t,x)|=\sup_{(t,x)\in[0,T]\times \mathbb{R}}|{r}^1(t,x)-{r}^2(t,x)|.
\end{align*}
Then, it is easy to check that 
\begin{equation}\label{diffK}
\begin{split}
\sup_{t\in[0,T]}|k^1_t-k^2_t|\leq &|\bar{k}^1_T-\bar{k}^2_T|+\sup_{t\in[0,T]}|\bar{k}^1_t-\bar{k}^2_t|\\
\leq & 2\frac{C}{c}|a^1-a^2|+4\frac{C}{c}\sup_{t\in[0,T]}|s^1_t-s^2_t|+\frac{2}{c}(\bar{L}_T\vee\bar{R}_T)
\end{split}
\end{equation}

\begin{remark}
Theorem \ref{BSP} is the extension of Lemma 2.5 in \cite{FS} to the case where the reflecting boundaries are nonlinear. Besides, suppose $\alpha^i,\beta^i\in C[0,T]$ satisfy that $\alpha^i\leq \beta^i$ and $\alpha^i_T\leq a^i\leq \beta^i_T$, $i=1,2$. Set $l^i(t,x)=x-\beta^i_t$, $r^i(t,x)=x-\alpha^i_t$, $i=1,2$. Let $(x^i,k^i)$ be the solution to the backward Skorokhod problem $\mathbb{BSP}_{l^i}^{r^i}(s^i,a^i)$, $i=1,2$. Then, estimate \eqref{diffK} coincides with Eq. (2.7) in \cite{FS}.
\end{remark}

\begin{remark}
Note that under the assumption as in Theorem \ref{BSP}, each component of the solution to the backward Skorokhod problem $\mathbb{BSP}_l^r(s,a)$ is continuous. Therefore, the functions $\{l(t,x_t)\}_{t\in[0,T]}$ and $\{r(t,x_t)\}_{t\in[0,T]}$ are continuous. Then, condition \eqref{iii} in Definition \ref{def} can be written as
\begin{align*}
\int_0^T l(s,x_s)dk^l_s=0, \ \int_0^T r(s,x_s)dk^r_s=0.
\end{align*}
\end{remark}

\section{BSDE with double mean reflections}

\subsection{Existence and uniqueness result}

In this subsection, we study the well-posedness of BSDEs with double mean reflections of the form \eqref{nonlinearyz}. Similar to the case with one reflecting boundary introduced in \cite{BEH}, we first construct a solution for the constant coefficient case and then deal with the general case by a Picard contraction argument. For this purpose, for any fixed $S\in \mathcal{S}^2$, we define the following two maps $r,l:[0,T]\times\mathbb{R}\rightarrow\mathbb{R}$ as
\begin{equation}\label{operator}
\begin{split}
r^S(t,x):=\E[R(t,x+S_t-\E[S_t])],\  l^S(t,x):=\E[L(t,x+S_t-\E[S_t])]. 
\end{split}
\end{equation}
For simplicity, we always omit the superscript $S$. By Assumption \ref{ass2}, $r$ and $l$ are well-defined. Besides, we show that $r,l$ satisfy Assumption \ref{ass1}, which ensures that they can be regarded as reflecting boundary functions for some backward Skorokhod problem.

\begin{lemma}\label{proofofass1}
Under Assumption \ref{ass2}, for any $S\in \mathcal{S}^2$, $l,r$ defined by \eqref{operator} satisfy Assumption \ref{ass1}.
\end{lemma}

\begin{proof}
We first show that $r$ satisfies Assumption \ref{ass1} (i)-(iii). In fact, it is easy to check that for any fixed $t\in[0,T]$, $r(t,\cdot)$ is strictly increasing. For any $x,y\in \mathbb{R}$ with $x<y$, we have
\begin{align*}
&r(t,y)-r(t,x)=\E[R(t,y+S_t-\E[S_t])-R(t,x+S_t-\E[S_t])]\leq C(y-x),\\
&r(t,y)-r(t,x)=\E[R(t,y+S_t-\E[S_t])-R(t,x+S_t-\E[S_t])]\geq c(y-x).
\end{align*}
Hence, $r$ satisfies Assumption \ref{ass1} (iii). Observe that 
\begin{align*}
|R(t,x+S_t-\E[S_t])|\leq \sup_{t\in[0,T]}(|R(t,0)|+C|S_t|+C\E[|S_t|])+C|x|.
\end{align*}
By the dominated convergence theorem, $r$ satisfies Assumption \ref{ass1} (i). Similarly, we can show that $l$ satisfies Assumption \ref{ass1} (i)-(iii). It remains to prove Assumption \ref{ass1} (iv) holds. Note that
\begin{align*}
r(t,x)-l(t,x)=\E[R(t,x+S_t-\E[S_t])-L(t,x+S_t-\E[S_t])]\geq \inf_{\omega,t,x}(R(\omega,t,x)-L(\omega,t,x))>0.
\end{align*}
The proof is complete.
\end{proof}

Now, we first focus on the BSDE with double mean reflections whose coefficient does not depend on $(Y,Z)$. 
\begin{proposition}\label{prop7}
Given $C\in \mathcal{H}^2$, the BSDE with double mean reflections
\begin{equation}\label{nonlinearnoyz}
\begin{cases}
Y_t=\xi+\int_t^T C_sds-\int_t^T Z_s dB_s+K_T-K_t, \\
\E[L(t,Y_t)]\leq 0\leq \E[R(t,Y_t)], \\
K_t=K^R_t-K^L_t, \int_0^T \E[R(t,Y_t)]dK_t^R=\int_0^T \E[L(t,Y_t)]dK^L_t=0,
\end{cases}
\end{equation}
has a unique solution $(Y,Z,K)\in \mathcal{S}^2\times \mathcal{H}^2\times BV[0,T]$.
\end{proposition}
\begin{proof}
We first prove the uniqueness. 
Let $(Y^i,Z^i,K^i)$ be two solutions to \eqref{nonlinearnoyz}, $i=1,2$. Set $X_t=\E_t[\xi+\int_t^T C_sds]$. Consequently, we have $Y_t^i=X_t+K^i_T-K^i_t$. Suppose that there exists $t_1<T$ such that
\begin{displaymath}
K^1_T-K^1_{t_1}>K^2_T-K^2_{t_1}.
\end{displaymath}
Set
\begin{displaymath}
t_2=\inf\{t\geq t_1: K_T^1-K^1_t=K_T^2-K^2_t\}.
\end{displaymath}
It is easy to check that
\begin{displaymath}
K_T^1-K^1_t> K^2_T-K^2_t, \ t_1\leq t<t_2.
\end{displaymath}
Noting that $R$ and $L$ are strictly increasing, we have for any $t_1\leq t<t_2$
\begin{align*}
&\E[R(t,X_t+K_T^1-K_t^1)]>\E[R(t,X_t+K_T^2-K_t^2)]\geq 0, \\
&\E[L(t,X_t+K_T^2-K_t^2)]<\E[L(t,X_t+K_T^1-K_t^1)]\leq 0.
\end{align*}
The flat-off condition
\begin{displaymath}
\int_{t_1}^{t_2}\E[R(t,X_t+K_T^1-K_t^1)]dK^{1,R}_t=\int_{t_1}^{t_2}\E[L(t,X_t+K_T^2-K_t^2)]dK^{2,L}_t=0
\end{displaymath}
implies that $dK^{1,R}_t=dK^{2,L}_t=0$ on the interval $[t_1,t_2]$. We obtain that
\begin{align*}
K_T^1-K_{t_2}^1=&K^1_T-(K^{1,R}_{t_2}-K^{1,L}_{t_2})=K^{1}_T-K^{1,R}_{t_1}+K^{1,L}_{t_2}\geq K^1_T-K^{1,R}_{t_1}+K^{1,L}_{t_1}\\
=&K_T^1-K^1_{t_1}>K^2_T-K^2_{t_1}=K_T^2-(K^{2,R}_{t_1}-K^{2,L}_{t_1})\\
=&K^2_T+K^{2,L}_{t_2}-K^{2,R}_{t_1}\geq K^2_T+K^{2,L}_{t_2}-K^{2,R}_{t_2}= K^2_T-K^2_{t_2},
\end{align*}
which contradicts the definition of $t_2$. Therefore, we have $K^1=K^2$. Note that $(Y^i+K^i,Z^i)$ may be seen as the solution to BSDE with terminal value $\xi+K^i_T$ and generator $g\equiv C$. By the uniqueness of solutions to BSDEs, we have $Y^1=Y^2$ and $Z^1=Z^2$.

Now we prove the existence.  For any $t\in[0,T]$, set 
\begin{align*}
&\widetilde{Y}_t=\xi+\int_t^T C_s ds-(M_T-M_t),\\
&s_t=\E[\int_0^t C_sds], \ a=\E[\xi],
\end{align*}
where $M_t=\E_t[\xi+\int_0^T C_s ds]-\E[\xi+\int_0^T C_s ds]$. Clearly, we have $s\in C[0,T]$ and $M$ is a square-integrable martingale. Therefore, there exists some $Z\in \mathcal{H}^2$, such that
\begin{align*}
M_t=\int_0^t Z_s dB_s.
\end{align*}
 For any $(t,x)\in[0,T]\times \mathbb{R}$, we define
\begin{align*}
l(t,x):=\E[L(t,\widetilde{Y}_t-\E[\widetilde{Y}_t]+x)], \ r(t,x):=\E[R(t,\widetilde{Y}_t-\E[\widetilde{Y}_t]+x)].
\end{align*}
By Lemma \ref{proofofass1}, $l,r$ satisfy Assumption \ref{ass1} and 
\begin{align*}
l(T,a)=\E[L(T,\xi)]\leq 0\leq \E[R(T,\xi)]=r(T,a).
\end{align*}
Therefore, the backward Skorokhod problem $\mathbb{BSP}_l^r(s,a)$ admits a unique solution $(x,K)$. Now, we set 
\begin{align*}
Y_t=\xi+\int_t^T C_s-(M_T-M_t)+(K_T-K_t)=\xi+\int_t^T C_sds-\int_t^T Z_s dB_s+K_T-K_t.
\end{align*}
We claim that $(Y,Z,K)$ is the solution to \eqref{nonlinearnoyz}. In fact, it is easy to check that 
\begin{align*}
\E[L(t,Y_t)]=&\E[L(t,\xi+\int_t^T C_sds-(M_T-M_t)+K_T-K_t)]\\
=&\E[L(t,\widetilde{Y}_t+x_t-\E[\xi+\int_t^T C_s ds])]\\
=&\E[L(t,\widetilde{Y}_t-\E[\widetilde{Y}_t]+x_t)]=l(t,x_t).
\end{align*}
Similarly, we have $\E[R(t,Y_t)]=r(t,x_t)$. Recalling that $(x,K)=\mathbb{BSP}_l^r(s,a)$, it follows that
\begin{align*}
&\E[L(t,Y_t)]=l(t,x_t)\leq 0\leq r(t,x_t)=\E[R(t,X_t)],\\
&\int_0^T \E[L(t,Y_t)]dK^L_t=\int_0^T l(t,x_t)dK^L_t=0,\\ 
&\int_0^T \E[R(t,Y_t)]dK^L_t=\int_0^T r(t,x_t)dK^R_t=0.
\end{align*} 
The proof is complete.
\end{proof}

Now, we are in a position to state the main result of this paper, i.e., the well-posedness of BSDEs with double mean reflections. Recall that the one of the main technical results in \cite{BEH} for the single reflected case is the Lipschitz continuity for the operator $R_t$ for the $L^1$-norm (see Lemma 8 in \cite{BEH}), where $R_t$ is defined as 
\begin{align*}
R_t: L^2(\mathcal{F}_T)\rightarrow [0,\infty), \ X\mapsto \inf\{x\geq 0:\E[R(t,x+X)]\geq 0\}.
\end{align*}
In fact, the Lipschitz continuity property of $R_t$ is mainly used to derive some estimates for the compensator function $K$ (see the proof of Theorem 9 in \cite{BEH}). In our cases, we have already obtained this estimate (see Eq. \eqref{diffK}), which plays an important role to construct a contraction mapping.

\begin{theorem}\label{main}
Given $\xi\in L^2(\mathcal{F}_T)$, suppose that $L,R$ satisfy Assumption \ref{ass2} with $\E[L(T,\xi)]\leq 0\leq \E[R(T,\xi)]$ and $f$ satisfy Assumption \ref{assf}. Then, the BSDE with double mean reflection \eqref{nonlinearyz} has a unique solution $(Y,Z,K)\in \mathcal{S}^2\times \mathcal{H}^2\times BV[0,T]$.
\end{theorem}

\begin{proof}
Given $U^i\in \mathcal{S}^2$, $V^i\in \mathcal{H}^2$, $i=1,2$, Proposition \ref{prop7} ensures that the following BSDE with double mean reflection admits a unique solution $(Y^i,Z^i,K^i)$,
\begin{displaymath}
\begin{cases}
Y^i_t=\xi+\int_t^T f(s,U^i_s,V^i_s)ds-\int_t^T Z^i_s dB_s+K^i_T-K^i_t, \\
\E[L(t,Y^i_t)]\leq 0\leq \E[R(t,Y^i_t)], \\
K^i_t=K^{i,R}_t-K^{i,L}_t, \int_0^T \E[R(t,Y^i_t)]dK_t^{i,R}=\int_0^T \E[L(t,Y^i_t)]dK^{i,L}_t=0.
\end{cases}
\end{displaymath}
We define
\begin{align*}
\hat{F}_t=F^1_t-F^2_t, \textrm{ where } F=Y,Z,K,U,V, \ \hat{f}_t=f(t,U^1_t,V^1_t)-f(t,U^2_t,V^2_t). 
\end{align*}
Simple calculation yields that 
\begin{align*}
\hat{Y}_t=\E_t[\int_t^T \hat{f}_sds]+\hat{K}_T-\hat{K}_t.
\end{align*}
Applying B-D-G inequality and the Lipschitz continuity of $f$, we obtain that there exists a constant $M(\lambda)$ depending on $\lambda$, such that 
\begin{align*}
\E[\sup_{t\in[0,T]}|\hat{Y}_t|^2]\leq M(\lambda)\E[(\int_0^T [|\hat{U}_t|+|\hat{V}_t|]dt)^2]+2\sup_{t\in[0,T]}|\hat{K}_t|^2.
\end{align*}
By the proof of Proposition \ref{prop7} and recalling \eqref{diffK}, we have
\begin{align*}
\sup_{t\in[0,T]}|\hat{K}_t|\leq  M(c,C)\{\sup_{t\in[0,T]}|s^1_t-s^2_t|+\sup_{(t,x)\in[0,T]\times\mathbb{R}}|l^1(t,x)-l^2(t,x)|\vee \sup_{(t,x)\in[0,T]\times\mathbb{R}}|r^1(t,x)-r^2(t,x)|\},
\end{align*}
where 
\begin{align*}
&s^i_t=\E[\int_0^t f(s,U^i_s,V^i_s)ds], \ \widetilde{Y}^i_t=\xi+\int_t^T f(s,U_s^i,V_s^i)ds-(M^i_T-M^i_t), \\
&M^i_t=\E_t[\xi+\int_0^T f(s,U^i_s,V^i_s)ds]-\E[\xi+\int_0^T f(s,U^i_s,V^i_s)ds],\\
&l^i(t,x)=\E[L(t,\widetilde{Y}^i_t-\E[\widetilde{Y}^i_t]+x)],\  r^i(t,x)=\E[R(t,\widetilde{Y}^i_t-\E[\widetilde{Y}^i_t]+x)].
\end{align*}
Then, we deduce that there exists a positive constant $M(c,C,\lambda)$, such that
\begin{align*}
\sup_{t\in[0,T]}|\hat{K}_t|^2\leq M(c,C,\lambda)\E[(\int_0^T [|\hat{U}_t|+|\hat{V}_t|]dt)^2].
\end{align*}
All the above analysis indicates that 
\begin{align*}
\E[\sup_{t\in[0,T]}|\hat{Y}_t|^2]\leq M(c,C,\lambda)\E[(\int_0^T [|\hat{U}_t|+|\hat{V}_t|]dt)^2].
\end{align*}

On the other hand, since
\begin{align*}
\int_0^T \hat{Z}_s dB_s=\int_0^T\hat{f}_s ds-\hat{Y}_0+\hat{K}_T-\hat{K}_0,
\end{align*}
we obtain that 
\begin{align*}
\E[\int_0^T |\hat{Z}_s|^2ds]\leq M(c,C,\lambda)\E[(\int_0^T [|\hat{U}_t|+|\hat{V}_t|]dt)^2].
\end{align*}
We finally deduce that 
\begin{align*}
\E[\sup_{t\in[0,T]}|\hat{Y}_t|^2+\int_0^T |\hat{Z}_s|^2ds]\leq &M(c,C,\lambda)\E[(\int_0^T [|\hat{U}_t|+|\hat{V}_t|]dt)^2]\\
\leq &M(c,C,\lambda)T\max(1,T)\E[\sup_{t\in[0,T]}|\hat{U}_t|^2+\int_0^T|\hat{V}_s|^2ds].
\end{align*}
Choosing $T$ small enough such that $M(c,C,\lambda)T\max(1,T)<1$, we have constructed a contraction mapping. Therefore, when $T$ small enough, the BSDE with double mean reflections \eqref{nonlinearyz} has a unique solution.

For the general case, let us choose $n\geq 1$ such that $\frac{1}{n^2}M(c,C,\lambda)T\max(n,T)<1$. For $i=0,1,\cdots,n$, set $T_i:=\frac{iT}{n}$. By backward induction, for $i=n,n-1,\cdots,1$, there exists a unique solution $(Y^i,Z^i,K^i)$ to the following BSDE with double mean reflection on the interval $[T_{i-1},T_i]$
\begin{displaymath}
\begin{cases}
Y^i_t=Y^{i+1}_{T_i}+\int_t^{T_i} f(s,Y^i_s,Z^i_s)ds-\int_t^{T_i} Z^i_s dB_s+K^i_{T_i}-K^i_t, \\
\E[L(t,Y^i_t)]\leq 0\leq \E[R(t,Y^i_t)],  t\in[T_{i-1},T_i]\\
K^i_{T_{i-1}}=0, \ K^i_t=K^{i,R}_t-K^{i,L}_t, \int_{T_{i-1}}^{T_i} \E[R(t,Y^i_t)]dK_t^{i,R}=\int_{T_{i-1}}^{T_i} \E[L(t,Y^i_t)]dK^{i,L}_t=0,
\end{cases}
\end{displaymath}
where $Y^{n+1}_T=\xi$. Let us define $(Y,Z,K)$ on $[0,T]$ by setting
\begin{align*}
Y_t=Y^1_0 I_{\{0\}}(t)+\sum_{i=1}^n Y^i_t I_{(T_{i-1},T_i]}(t), \ Z_t=\sum_{i=1}^n Z^i_t I_{(T_{i-1},T_i]}(t),
\end{align*}
and $K_t=K^1_t$ on $[T_0,T_1]$ and for $i=2,\cdots,n$, $K_t=K^i_t+K_{T_{i-1}}$ on $[T_{i-1},T_i]$ ($K^R,K^L$ are defined similarly). It is straightforward to check that $(Y,Z,K)$ is a solution to \eqref{nonlinearyz}. Uniqueness is a direct consequence from the uniqueness on each small interval. The proof is complete. 
\end{proof}

\subsection{Properties of solution to BSDEs with double mean reflections}

Recall that the Skorokhod condition, also called the flat-off condition, ensures the minimality of the first component of the solution to classical reflected BSDEs and BSDE with mean reflection. It is worth pointing out that for the mean reflected case, the minimality condition only holds for some typical coefficient since the constraint is given in expectation instead of pointwisely in this case.

For the double mean reflected case, consider the following two mean reflected BSDEs with single flat-off condition
\begin{equation}\label{nonlinearyz1}
\begin{cases}
Y_t=\xi+\int_t^T f(s,Y_s,Z_s)ds-\int_t^T Z_s dB_s+K_T-K_t, \\
\E[L(t,Y_t)]\leq 0\leq \E[R(t,Y_t)], \\
K_t=K^R_t-K^L_t,\ K^R,K^L\in I[0,T], \ \int_0^T \E[R(t,Y_t)]dK_t^R=0,
\end{cases}
\end{equation} 
and 
\begin{equation}\label{nonlinearyz2}
\begin{cases}
Y_t=\xi+\int_t^T f(s,Y_s,Z_s)ds-\int_t^T Z_s dB_s+K_T-K_t, \\
\E[L(t,Y_t)]\leq 0\leq \E[R(t,Y_t)], \\
K_t=K^R_t-K^L_t, \ K^R,K^L\in I[0,T], \ \int_0^T \E[L(t,Y_t)]dK^L_t=0.
\end{cases}
\end{equation}
Roughly speaking, \eqref{nonlinearyz1} only imposes flat-off condition on the lower obstacle while \eqref{nonlinearyz2} only imposes flat-off condition on the upper obstacle. Therefore, \eqref{nonlinearyz1} only ensures the minimality condition on the force which to push the solution upwards while \eqref{nonlinearyz2}  only ensures the minimality condition on the force which to pull the solution downwards. It is natural to conjecture that the solution of BSDE with double mean reflection should lie between the solution of \eqref{nonlinearyz1} and \eqref{nonlinearyz2}.

\begin{proposition}\label{prop11}
Suppose that the coefficient $f$ satisfying Assumption \ref{assf} is of the following form
\begin{align}\label{equation24}
f:(t,y,z)\mapsto a_t y+h(t,z),
\end{align}
where $a$ is a deterministic and bounded measurable function. Let $\xi\in L^2(\mathcal{F}_T)$ and $L,R$ satisfy Assumption \ref{ass2} with $\E[L(T,\xi)]\leq 0\leq \E[R(T,\xi)]$. Suppose that $(\underline{Y},\underline{Z},\underline{K})$, $(\bar{Y},\bar{Z},\bar{K})$, $(Y,Z,K)$ is a solution to \eqref{nonlinearyz1}, \eqref{nonlinearyz2}, \eqref{nonlinearyz}, respectively. Then, for any $t\in[0,T]$, we have $\underline{Y}_t\leq Y_t\leq \bar{Y}_t$.
\end{proposition}

\begin{proof}
It suffices to prove the second inequality since the first one can be proved analogously. To this end, we first consider the case where $f$ does not depends on $y$. In this case, it is clear that $(Y-(K_T-K),Z)$ and $(\bar{Y}-(\bar{K}_T-\bar{K}),\bar{Z})$ are solutions to classical BSDE with terminal value $\xi$ and coefficient $f$. By the uniqueness result for BSDEs, we have
\begin{align}\label{equation25}
Y_t-(K_T-K_t)=\bar{Y}_t-(\bar{K}_T-\bar{K}_t), \ t\in[0,T]. 
\end{align} 
It remains to show that $K_T-K\leq \bar{K}_T-\bar{K}$. Suppose that there exists a $t_1<T$, such that 
\begin{align*}
K_T-K_{t_1}> \bar{K}_T-\bar{K}_{t_1}.
\end{align*}
Set 
\begin{align*}
t_2=\inf\{t\geq t_1: K_T-K_t\leq \bar{K}_T-\bar{K}_t\}.
\end{align*}
By the continuity of $K$ and $\bar{K}$, we have 
\begin{align}\label{equation26}
K_T-K_{t_2}= \bar{K}_T-\bar{K}_{t_2}, \ K_T-K_t> \bar{K}_T-\bar{K}_t, \ t\in [t_1,t_2).
\end{align}
Combining with Eq. \eqref{equation25}, we obtain that $Y_t>\bar{Y}_t$, $t\in[t_1,t_2)$. It follows that for any $t\in[t_1,t_2)$
\begin{align*}
\E[R(t,Y_t)]>\E[R(t,\bar{Y}_t)]\geq 0\geq \E[L(t,{Y}_t)]>\E[L(t,\bar{Y}_t)].
\end{align*}
Due to the flat-off condition, we deduce that $d K^R_t=d \bar{K}^L_t=0$, $t\in[t_1,t_2)$. Recalling \eqref{equation26} and noting that $K^L$ is nondecreasing, we have
\begin{align*}
&\bar{K}_T-(\bar{K}^R_{t_1}-\bar{K}^L_{t_2})=\bar{K}_T-(\bar{K}^R_{t_1}-\bar{K}^L_{t_1})=\bar{K}_T-\bar{K}_{t_1}\\
<&K_T-K_{t_1}=K_T-(K^R_{t_1}-K^L_{t_1})=K_T-(K^R_{t_2}-K^L_{t_1})\leq K_T-(K^R_{t_2}-K^L_{t_2})\\
=&K_T-K_{t_2}=\bar{K}_T-\bar{K}_{t_2}=\bar{K}_T-(\bar{K}^R_{t_2}-\bar{K}^L_{t_2}),
\end{align*}
which implies that $\bar{K}^R_{t_1}>\bar{K}^R_{t_2}$, contradicting to the fact that $\bar{K}^R$ is nondecreasing.

For the case that $f$ takes the form \eqref{equation24}, set $A_t:=\int_0^t a_s ds$, $t\in[0,T]$ and define
\begin{align*}
\tilde{Y}_t=e^{A_t} Y_t,  \ \tilde{Z}_t=e^{A_t} Z_t, \ \tilde{K}_t=e^{A_t} K_t, \ \tilde{K}^R_t=e^{A_t} K^R_t, \ \tilde{K}^L_t=e^{A_t} K^L_t.
\end{align*}
It is easy to check that $(\tilde{Y},\tilde{Z},\tilde{K})$ is the solution to BSDE with double mean reflections associated to parameters
\begin{align*}
\tilde{\xi}=e^{A_T}\xi, \ \tilde{f}(t,z)=e^{A_t}f(t,e^{-A_t}z), \ \tilde{L}(t,y)=L(t,e^{-A_t}y), \ \tilde{R}(t,y)=R(t,e^{-A_t}y).
\end{align*}
Then, we have transformed the problem to the case that the coefficient does not depend on $y$. The proof is complete.
\end{proof}

\begin{remark}
(i) Proposition \ref{prop11} provides an alternative proof of the uniqueness result to the BSDE with double mean reflections whose coefficient is of form \eqref{equation24}. 

\noindent (ii) Suppose that $L=-\infty$. Proposition \ref{prop11} degenerates to Theorem 11 in \cite{BEH}. That is, a deterministic flat solution to a BSDE with lower mean reflection is minimal among all the deterministic solutions. Similarly, supposing that $R=+\infty$, we conclude that a deterministic flat solution to a BSDE with upper mean reflection is maximal among all the deterministic solutions.
\end{remark}

Let us recall that the first component of solution to a doubly reflected BSDE coincides with the value function of an appropriate Dynkin game. It is natural to consider if the solution to a BSDE with double mean reflections corresponds to some optimization problem. Fortunately, the answer is affirmative.  Let $(Y,Z,K)$ be the solution to the BSDE with double mean reflection \eqref{nonlinearyz}. We define
\begin{align*}
\bar{Y}_t:=\E_t[\xi+\int_t^T f(s,Y_s,Z_s)ds].
\end{align*}
It is easy to check that 
\begin{align*}
Y_t=\bar{Y}_t-\E[\bar{Y}_t]+\E[Y_t].
\end{align*}
Recalling Lemma \ref{proofofass1} and noting that $\bar{Y}\in \mathcal{S}^2$, $r^{\bar{Y}}(t,\cdot),l^{\bar{Y}}(t,\cdot)$ are continuous and strictly increasing, $t\in[0,T]$, where $r^{\bar{Y}},l^{\bar{Y}}$ are defined in \eqref{operator}. Besides, we may check that 
\begin{align*}
\lim_{x\rightarrow \infty}r^{\bar{Y}}(t,x)=\lim_{x\rightarrow \infty}l^{\bar{Y}}(t,x)=+\infty, \ 
\lim_{x\rightarrow -\infty}r^{\bar{Y}}(t,x)=\lim_{x\rightarrow -\infty}l^{\bar{Y}}(t,x)=-\infty.
\end{align*}
Therefore, for any fixed $t\in[0,T]$, each of the following equation admits a unique solution
\begin{align*}
\E[R(t,\bar{Y}_t-\E[\bar{Y}_t]+x]=0, \ \E[L(t,\bar{Y}_t-\E[\bar{Y}_t]+x]=0,
\end{align*}
which are denoted by $\bar{r}_t$ and $\bar{l}_t$, respectively. 

\begin{theorem}\label{theorem3.6}
Suppose that $(Y,Z,K)$ is the solution to the BSDE with double mean reflections \eqref{nonlinearyz}. Then, for  any $t\in[0,T]$, we have
\begin{align*}
\E[Y_t]=\sup_{q\in[t,T]}\inf_{s\in[t,T]}R_t(s,q)=\inf_{s\in[t,T]}\sup_{q\in[t,T]}R_t(s,q),
\end{align*}
where 
\begin{align*}
R_t(s,q)=\E[\int_t^{s\wedge q} f(u,Y_u,Z_u)du+\xi I_{\{s\wedge q=T\}}]+\bar{r}_q I_{\{q< T,q\leq s\}}+\bar{l}_s I_{\{s<q\}}.
\end{align*}
\end{theorem}

\begin{proof}
Fix $t\in[0,T]$. It suffices to show that for any $\varepsilon>0$, there exist $s^\varepsilon_t,q^\varepsilon_t\in[t,T]$, such that for all $s,q\in[t,T]$,
\begin{equation}\label{equa3.2}
-\varepsilon+R_t(s^\varepsilon_t,q)\leq \E[Y_t]\leq R_t(s,q^\varepsilon_t)+\varepsilon.
\end{equation}

First, noting that for all $t\in[0,T]$
\begin{align*}
&\E[R(t,\bar{Y}_t-\E[\bar{Y}_t]+\bar{r}_t)]=0\leq \E[R(t,Y_t)]=\E[R(t,\bar{Y}_t-\E[\bar{Y}_t]+\E[Y_t])], \\
&\E[L(t,\bar{Y}_t-\E[\bar{Y}_t]+\bar{l}_t)]=0\geq \E[L(t,Y_t)]=\E[L(t,\bar{Y}_t-\E[\bar{Y}_t]+\E[Y_t])],
\end{align*}
it follows that 
\begin{align}\label{barrbarl}
\bar{r}_t\leq \E[Y_t]\leq \bar{l}_t, \ t\in[0,T].
\end{align}
 Set 
\begin{align*}
s^\varepsilon_t=\inf\{s>t:\E[Y_s]\geq \bar{l}_s-\varepsilon\}\wedge T.
\end{align*}
Since $\E[Y_s]<\bar{l}_s-\varepsilon$ on $s\in(t,s^\varepsilon_t)$, we have
\begin{align*}
\E[L(s,Y_s)]=\E[L(s,\bar{Y}_s-\E[\bar{Y}_s]+\E[Y_s])]<\E[L(s,\bar{Y}_s-\E[\bar{Y}_s]+\bar{l}_s)]=0,
\end{align*}
which implies that $K^L_q-K^L_t=0$ for $q\in(t,s^\varepsilon_t)$ and thus $K_q-K_t=K^R_q-K^R_t\geq 0$ for $q\in(t,s^\varepsilon_t)$. By the continuity of $K$, $K_q-K_t\geq 0$ holds for $q\in(t,s^\varepsilon_t]$. Simple calculation yields that for $q\in(t,s^\varepsilon_t]$
\begin{align*}
R_t(s^\varepsilon_t,q)=&\E[\int_t^ q f(s,Y_s,Z_s)ds+\xi I_{\{q=T\}}]+\bar{r}_q I_{\{q<T\}}\\
\leq &\E[\int_t^ q f(s,Y_s,Z_s)ds+\xi I_{\{q=T\}}]+\E[Y_q] I_{\{q<T\}}+K_q-K_t\\
=&\E[Y_t+\int_t^q Z_s dB_s]=\E[Y_t],
\end{align*}
where we have used \eqref{barrbarl} in this inequality. On the other hand, for $q\in(s^\varepsilon_t,T]$,  we have
\begin{align*}
R_t(s^\varepsilon_t,q)=&\E[\int_t^{s^\varepsilon_t} f(s,Y_s,Z_s)ds]+\bar{l}_{s^\varepsilon_t}\\
\leq &\E[\int_t^{s^\varepsilon_t}  f(s,Y_s,Z_s)ds]+\E[Y_{s^\varepsilon_t} ] +\varepsilon+K_{s^\varepsilon_t} -K_t\\
=&\E[Y_t+\int_t^{s^\varepsilon_t}  Z_s dB_s]+\varepsilon=\E[Y_t]+\varepsilon,
\end{align*}
where we have used \eqref{barrbarl} and the definition of $s^\varepsilon_t$ in this inequality. All the above analysis implies the first inequality in \eqref{equa3.2}. Set
\begin{align*}
q^\varepsilon_t=\inf\{s>t: \E[Y_s]\leq \bar{r}_s+\varepsilon\}\wedge T.
\end{align*}
By a similar analysis as above, we could obtain the second inequality in \eqref{equa3.2}. The proof is complete.
\end{proof}

\begin{remark}
Theorem \ref{theorem3.6} is a generalization of Theorem 3.6 in \cite{FS} to the case when the two reflecting constraints are nonlinear.
\end{remark}

\section{Construction by penalization in a special case}

For the classical reflected BSDEs with both single obstacle and double obstacles, an effective method to construct solutions is approximation via penalization. However, for the mean reflected case, the constraint only integrates the distribution of $Y$ but not the pointwise value of $Y$. Therefore, the classical proof is invalid. However, Briand et al. show that the solution of BSDE with linear mean reflection is the limit of penalized mean-field BSDEs (see Proposition 6 in \cite{BEH}). In this section, we apply an analogous penalization method to construct the solution to the following BSDE with two linear reflections whose parameters are given by $(\xi,f,l,r)$: 
\begin{equation}\label{linearcase}
\begin{cases}
Y_t=\xi+\int_t^T f(s,Y_s,Z_s)ds-\int_t^T Z_s dB_s+K_T-K_t, \\
l_t\leq \E[Y_t]\leq r_t, \\
K_t=K^l_t-K^r_t, \int_0^T (\E[Y_t]-l_t)dK_t^l=\int_0^T (r_t-\E[Y_t])dK^r_t=0,
\end{cases}
\end{equation}
where $K^l,K^r\in I[0,T]$. 


Before establishing the existence result by approximation via penalization, we present some a priori estimates similar with the classical reflected BSDEs, which will provide a different proof for the uniqueness result for BSDEs with double mean reflections.
\begin{proposition}\label{uniquenesslinear}
Suppose that $f^i$ satisfy Assumption \ref{assf}, $i=1,2$. Given $l,r\in C[0,T]$ and $\xi^i\in L^2(\mathcal{F}_T)$ with $l_T\leq \E[\xi]\leq r_T$, let $(Y^i,Z^i,K^i)$ be the solution to the BSDE with two linear reflections $(\xi^i,f^i,l,r)$. Then, there exists a constant $C$ depending on $\lambda,T$, such that 
\begin{align*}
\E[\int_0^T |\hat{Y}_t|^2dt]+\E[\int_0^T |\hat{Z}_t|^2 dt]\leq C\E[|\hat{\xi}|^2+\int_0^T |\hat{f}_t|^2dt],
\end{align*}
where $\hat{Y}_t=Y^1_t-Y^2_t$, $\hat{Z}_t=Z^1_t-Z^2_t$, $\hat{\xi}=\xi^1-\xi^2$ and $\hat{f}_t=f^1(t,Y^1_t,Z^1_t)-f^2(t,Y^1_t,Z_t^1)$.
\end{proposition}

\begin{proof}
 Let $\hat{K}_t=K^1_t-K^2_t$ and $\hat{f}^2_t=f^2(t,Y^1_t,Z^1_t)-f(t,Y^2_t,Z_t^2)$. Applying It\^{o}'s formula to $e^{at}\hat{Y}_t^2$, where $a$ is a positive constant to be determined later, we have
\begin{equation}
\begin{split}\label{e1}
&\hat{Y}_t^2 e^{at}+\int_t^T ae^{as}\hat{Y}_s^2 ds+\int_t^T e^{as} \hat{Z}_s^2 ds\\
=&e^{aT}|\hat{\xi}|^2+\int_t^T 2e^{as}\hat{Y}_s (\hat{f}_s+\hat{f}^2_s) ds+\int_t^T 2e^{as}\hat{Y}_s d\hat{K}_s-2\int_t^T e^{as}\hat{Y}_s\hat{Z}_sdB_s.
\end{split}
\end{equation}
By the assumption on $f^2$ and the H\"{o}lder inequality, we obtain that
\begin{equation}
\begin{split}\label{e2}
\int_t^T 2e^{as}\hat{Y}_s \hat{f}^2_s ds&\leq \int_t^T 2\lambda e^{as}(\hat{Y}_s^2+|\hat{Y}_s\hat{Z}_s|) ds\leq \int_t^T (2\lambda+2\lambda^2)e^{as}\hat{Y}_s^2 ds+\frac{1}{2}\int_t^T e^{as}\hat{Z}_s^2 ds,\\
\int_t^T 2e^{as}\hat{Y}_s \hat{f}_s ds&\leq \int_t^T e^{as}\hat{Y}_s^2 ds+\int_t^T e^{as}|\hat{f}_s|^2 ds.
\end{split}
\end{equation}
Noting the flat-off condition, it is easy to check that
\begin{equation}\label{e3}
\begin{split}
\E[\int_t^T e^{as}\hat{Y}_s d\hat{K}_s]&=\int_t^T e^{as}\big((\E[Y_s^1]-l_s)-(\E[Y_s^2]-l_s)\big)dK^{1,l}_s\\
&+\int_t^T e^{as}\big((r_s-\E[Y_s^1])-(r_s-\E[Y_s^2])\big)dK^{1,r}_s\\
&-\int_t^T e^{as}\big((\E[Y_s^1]-l_s)-(\E[Y_s^2]-l_s)\big)dK^{2,l}_s\\
&-\int_t^T e^{as}\big((r_s-\E[Y_s^1])-(r_s-\E[Y_s^2])\big)dK^{2,r}_s\leq 0.
\end{split}
\end{equation}
Set $a=2(1+\lambda+\lambda^2)$. Combing Eqs. \eqref{e1}-\eqref{e3}, we get the desired result.
\end{proof}

We are now in a position to prove the existence result by a penalization method. For this purpose, we propose the following condition on the obstacles:
\begin{itemize}
\item[($H_{rl}$)] $r_t=\int_0^t a_s ds$, $l_t=\int_0^t b_s ds$ with $\int_0^T|a_t|^2 dt <\infty$, $\int_0^T|b_t|^2 dt<\infty$ and $l_t\leq r_t$, $t\in[0,T]$.
\end{itemize}

 Consider the following penalized mean-field BSDE
\begin{equation}\label{panelization}
Y_t^n=\xi+\int_t^T f(s,Y_s^n,Z_s^n)ds+\int_t^T n(\E[Y_s^n]-l_s)^-ds-\int_t^T n(\E[Y_s^n]-r_s)^+ ds-\int_t^T Z_s^n dB_s.
\end{equation}
By Theorem 3.1 in \cite{BLP}, the above equation admits a unique pair of solution $(Y^n,Z^n)\in \mathcal{S}^2\times \mathcal{H}^2$. Set $K^{n,l}_t=\int_0^t n(\E[Y_s^n]-l_s)^- ds$, $K^{n,r}_t=\int_0^t n(\E[Y_s]-r_s)^+ds$ and $K_t^n=K^{n,l}_t-K^{n,r}_t$. In the following, we show that $(Y^n,Z^n,K^n)$ converges to $(Y,Z,K)$, which is the solution to the BSDE with double linear mean reflections. In the following of this section, $C$ will always be a positive constant  depending on $T,\lambda$. We first establish the estimates for $Y^n$ and $Z^n$ uniformly in $n$. 

\begin{proposition}\label{estimateYnZn}
There exists a constant $C$ independent of $n$, such that
\begin{align*}
&\sup_{t\in[0,T]}\E[|Y_t^n|^2]\leq C(\int_0^T a_s^2ds+\E[\xi^2]+\E[\int_0^T |f(t,0,0)|^2 dt]),\\
&\E[\int_0^T |Z_s^n|^2 ds]\leq C(\int_0^T a_s^2ds+\E[\xi^2]+\E[\int_0^T |f(t,0,0)|^2 dt]).
\end{align*}
\end{proposition}

\begin{proof}
Set $\widetilde{Y}^n_t=Y^n_t-r_s$. Applying It\^{o}'s formula to $e^{\beta t}|\widetilde{Y}^n_t|^2$, where $\beta$ is a positive constant to be determined later, we have
\begin{equation}\label{eq1.61}\begin{split}
&e^{\beta t}|\widetilde{Y}^n_t|^2+\int_t^T \beta e^{\beta s}|\widetilde{Y}^n_s|^2 ds+\int_t^T e^{\beta s}|Z_s^n|^2 ds\\
=&e^{\beta T}|\xi-r_T|^2+\int_t^T 2e^{\beta s}\widetilde{Y}^n_s(f(s,Y_s^n,Z_s^n)+a_s)ds-\int_t^T 2e^{\beta s}\widetilde{Y}_s^n Z_s^ndB_s\\
&+\int_t^T 2ne^{\beta s}\widetilde{Y}_s^n(\E[Y_s^n]-l_s)^-ds-\int_t^T 2ne^{\beta s}\widetilde{Y}_s^n(\E[\widetilde{Y}_s^n])^+ ds.
\end{split}\end{equation}
It is easy to check that 
\begin{equation}\label{eq1.62}\begin{split}
2\widetilde{Y}^n_s(f(s,Y_s^n,Z_s^n)+a_s)\leq &|f(s,r_s,0)|^2+|a_s|^2+\frac{1}{2}|Z^n_s|^2+(2+2\lambda+2\lambda^2)|\widetilde{Y}^n_s|^2\\
\leq &|a_s|^2+2|f(s,0,0)|^2+2\lambda^2|r_s|^2+\frac{1}{2}|Z^n_s|^2+2(1+\lambda+\lambda^2)|\widetilde{Y}^n_s|^2
\end{split}\end{equation}
Set $\beta=2+2\lambda+2\lambda^2$. Plugging Eq. \eqref{eq1.62} to Eq. \eqref{eq1.61} and taking expectations on both sides, we have
\begin{align*}
\E[|\widetilde{Y}_t^n|^2+\int_t^T |Z^n_s|^2 ds]\leq &C\E[|\xi-r_T|^2+\int_t^T |f(s,0,0)|^2+\int_t^T r_s^2ds+\int_t^T a_s^2ds]\\
\leq &C\E[|\xi|^2+\int_0^T |f(s,0,0)|^2+\int_0^T a_s^2ds],
\end{align*}
where we have used the following facts
\begin{displaymath}
\E[\widetilde{Y}^n_s](\E[Y_s^n]-l_s)^-\leq 0,\ \E[\widetilde{Y}^n_s](\E[\widetilde{Y}_s^n])^+\geq 0.
\end{displaymath}
Recalling that $Y_t^n=\widetilde{Y}^n_t+r_t$, we obtain the desired result.
\end{proof}

\begin{proposition}\label{estimateYn-r}
There exists a constant $C$ independent of $n$, such that
\begin{align*}
n^2\int_0^T|(\E[Y_t^n]-r_t)^+|^2 dt\leq {C},\
n^2\int_0^T|(\E[Y_t^n]-l_t)^-|^2 dt\leq {C}.
\end{align*}
\end{proposition}

\begin{proof}
We only prove the first inequality since the second one can be proved similarly. Set $y_t^n=\E[Y^n_t]$. Taking expectations on both sides of Eq. \eqref{panelization}, we have
\begin{align*}
y_t^n=\E[\xi]+\int_t^T \E[f(s,Y^n_s,Z^n_s)]ds+\int_t^T n(y^n_s-l_s)^- ds-\int_t^T n(y^n_s-r_s)^+ds.
\end{align*} 
It is easy to check that 
\begin{align*}
d|(y^n_t-r_t)^+|^2=-2(y^n_t-r_t)^+ [ \E[f(t,Y^n_t,Z^n_t)]+a_s+n(y^n_t-l_t)^- -n(y^n_t-r_t)^+].
\end{align*}
Noting that $y^n_T-r_T\leq 0$, we have
\begin{align*}
&|(y^n_0-r_0)^+|^2+2n\int_0^T |(y^n_t-r_t)^+|^2dt\\
=&2\int_0^T (y^n_t-r_t)^+(\E[f(t,Y^n_t,Z^n_t)]+a_s) dt+2n\int_0^T (y^n_t-r_t)^+(y^n_t-l_t)^-dt\\
\leq &n\int_0^T  |(y^n_t-r_t)^+|^2dt+\frac{1}{n}\int_0^T (\E[f(t,Y^n_t,Z^n_t)]+a_t) ^2dt\\
\leq &n\int_0^T  |(y^n_t-r_t)^+|^2dt+\frac{C}{n}\int_0^T (\E[|f(t,0,0)|^2+|Y_t^n|^2+|Z_t^n|^2]+a^2_t)dt.
\end{align*}
Applying Proposition \ref{estimateYnZn}, we obtain the desired result.
\end{proof}

By Proposition \ref{estimateYnZn} and \ref{estimateYn-r}, applying the B-D-G inequality, we obtain the following more accurate estimate for $Y^n$.
\begin{proposition}\label{estimateKn}
There exists a constant $C$ independent of $n$, such that
\begin{align*}
\E[\sup_{t\in[0,T]}|Y_t^n|^2]\leq C.
\end{align*}
\end{proposition}


Now, we state the main theorem in this section.
\begin{theorem}
Suppose that $f$ satisfies Assumption \ref{assf} and $r,l$ satisfy condition ($H_{lr}$). Given $\xi\in L^2(\mathcal{F}_T)$ with $l_T\leq \E[\xi]\leq r_T$, the BSDE with double mean reflection \ref{linearcase} has a unique solution $(Y,Z,K)$. Furthermore, $(Y,Z,K)$ is the limit of $(Y^n,Z^n,K^n)$.
\end{theorem}

\begin{proof}
Uniqueness is a direct consequence of Proposition \ref{uniquenesslinear}. It remains to show the existence. We first prove that $(Y^n,Z^n,K^n)$ converges strongly to $(Y,K,Z)$. For this purpose, for any positive integers $n$,  set $\hat{X}:=X^{n+1}-X^n$ for $X=Y,Z,K$. By a similar analysis as Equations \eqref{e1} and \eqref{e2}, we have
\begin{align*}
e^{at}\hat{Y}_t^2+\frac{1}{2}\int_t^T e^{as}(\hat{Y}_s^2+\hat{Z_s^2})ds\leq 2\int_t^T e^{as}\hat{Y}_sd\hat{K}_s-2\int_t^T e^{as}\hat{Y}_s\hat{Z}_sdB_s,
\end{align*}
where $a=\frac{1}{2}+2(\lambda+\lambda^2)$, which implies that
\begin{equation}\label{e30}
\sup_{t\in[0,T]}\E[\hat{Y}_t^2]+\E[\int_0^T (\hat{Y}_s^2+\hat{Z_s^2})ds]\leq C\sup_{t\in[0,T]}\E[\int_t^T e^{as}\hat{Y}_sd\hat{K}_s].
\end{equation}
We denote $y^n_t=\E[Y^n_t]$, $v^n_t=(y^n_t-l_t)^-$ and $u^n_t=(y_t^n-r_t)^+$. It is easy to check that
\begin{equation}\label{eq1.81}\begin{split}
\E[\int_t^T e^{as}\hat{Y}_sd\hat{K}_s]=&\int_t^T e^{as}(y^{n+1}_s-y_s^n)\bigg(\big[(n+1)v_s^{n+1}-nv_s^n\big]-\big[(n+1)u_s^{n+1}-nu_s^n\big]\bigg)ds\\
 =&\int_t^T e^{as}[(y^{n+1}_s-l_s)-(y^n_s-l_s)]((n+1)v_s^{n+1}-nv_s^n)ds\\
 &-\int_t^T e^{as}[(y^{n+1}_s-r_s)-(y^n_s-r_s)]((n+1)u_s^{n+1}-nu_s^n)ds.
\end{split}\end{equation}
Note that for any $x,y\in\mathbb{R}$, we have
\begin{align*}
-nx^2+(2n+1)xy-(n+1)y^2=-n(x-\frac{2n+1}{2n}y)^2+\frac{y^2}{4n}.
\end{align*}
Simple calculation implies that 
\begin{equation}\label{eq1.82}\begin{split}
&[(y^{n+1}_s-l_s)-(y^n_s-l_s)]((n+1)v_s^{n+1}-nv_s^n)\\
\leq&-n|v_s^n|^2+(2n+1)v^n_s v^{n+1}_s-(n+1)|v_s^{n+1}|^2\leq \frac{|v^{n+1}_s|^2}{4n}
\end{split}\end{equation}
and
\begin{equation}\label{eq1.83}\begin{split}
&-[(y^{n+1}_s-r_s)-(y^n_s-r_s)]((n+1)u_s^{n+1}-nu_s^n)\\
\leq&-n|u_s^n|^2+(2n+1)u^n_s u^{n+1}_s-(n+1)|u_s^{n+1}|^2\leq \frac{|u^{n+1}_s|^2}{4n}.
\end{split}\end{equation}
Combining Eq. \eqref{eq1.81}-\eqref{eq1.83}, by Proposition \ref{estimateYn-r}, we have
\begin{align*}
\E[\int_t^T e^{as}\hat{Y}_sd\hat{K}_s]\leq \frac{C}{n}\int_0^T (|v^{n+1}_s|^2+|u^{n+1}_s|^2)ds\leq \frac{C}{n^3}.
\end{align*}
Plugging this estimate in \eqref{e30}, we have
\begin{align}\label{eq1.84}
\sup_{t\in[0,T]}\E[\hat{Y}_t^2]+\E[\int_0^T (\hat{Y}_s^2+\hat{Z_s^2})ds]\leq \frac{C}{n^3}.
\end{align}
Recalling Eq. \eqref{panelization} and the fact that $K^n$ is deterministic, it follows that
\begin{align*}
\hat{K}_T-\hat{K_t}=\E[\hat{Y}_t]-\E[\int_t^T (f(s,Y^{n+1}_s,Z^{n+1}_s)-f(s,Y_s^n,Z_s^n))ds].
\end{align*}
Applying Eq. \eqref{eq1.84} implies that
\begin{align*}
\sup_{t\in[0,T]}|\hat{K}_T-\hat{K_t}|^2\leq \frac{C}{n^3}.
\end{align*}
Finally, note that 
\begin{align*}
\hat{Y}_t=\E_t[\int_t^T (f(s,Y^{n+1}_s,Z^{n+1}_s)-f(s,Y_s^n,Z_s^n))ds]+\hat{K}_T-\hat{K_t}.
\end{align*}
By the B-D-G inequality, all the above analysis implies that there exists a triple $(Y,Z,K)\in\mathcal{S}^2\times \mathcal{H}^2\times BV[0,T]$ such that
\begin{align}\label{equation6.25}
\E[\sup_{t\in[0,T]}|Y_t^n-Y_t|^2]+\E[\int_0^T |Z_s^n-Z_s|^2 ds]+\sup_{t\in[0,T]}|K^n_t-K_t|\rightarrow 0, \textrm{ as } n\rightarrow \infty.
\end{align}

Now, we a in a position to show that the limit processes $(Y,Z,K)$ is indeed the solution to \eqref{linearcase}. By Proposition \ref{estimateYn-r}, it is easy to check that $l_t\leq \E[Y_t]\leq r_t$, $t\in[0,T]$. It remains to prove that there exists $K^l$, $K^r\in I[0,T]$, such that
\begin{displaymath}
K_t=K^l_t-K^r_t, \ \int_0^T (\E[Y_t]-l_t)dK^l_t=\int_0^T (r_t-\E[Y_t])dK^r_t=0.
\end{displaymath}
 By Proposition \ref{estimateYn-r}, the nonnegative measurable functions $\{n v^n\}_{n\in\mathbb{N}}$ and $\{n u^n\}_{n\in\mathbb{N}}$ are bounded in $L^2([0,T])$. Therefore, there exist two nonnegative measurable functions $v,u$ such that (along a relabelled subsequence),
\begin{align}\label{equation6.26}
n v^n\rightarrow v, \ n u^n\rightarrow u, \textrm{ as } n\rightarrow \infty, \textrm{ weakly in } L^2([0,T]).
\end{align}
Simple calculation implies that 
\begin{align*}
&|\int_0^T (\E[Y^n_t]-l_t)nv^n_t dt-\int_0^T (\E[Y_t]-l_t)v_t dt|\\
=&|\int_0^T (\E[Y^n_t]-\E[Y_t])nv^n_t dt+\int_0^T (\E[Y_t]-l_t)(nv^n_t-v_t) dt|\\
\leq &\E[\sup_{t\in[0,T]}|Y^n_t-Y_t|]\int_0^T nv^n_t dt+|\int_0^T (\E[Y_t]-l_t)(nv^n_t-v_t) dt|.
\end{align*}
By Eq. \eqref{equation6.25}, \eqref{equation6.26} and Proposition \ref{estimateYn-r}, we obtain that
\begin{align*}
\lim_{n\rightarrow \infty}|\int_0^T (\E[Y^n_t]-l_t)nv^n_t dt-\int_0^T (\E[Y_t]-l_t)v_t dt|=0.
\end{align*}
Since $\E[Y_t]\geq l_t$, $t\in[0,T]$ and $v$ is nonnegative, we have 
\begin{align*}
\int_0^T (\E[Y_t]-l_t)v_t dt\geq 0\geq \int_0^T (\E[Y^n_t]-l_t)n(\E[Y_t^n]-l_t)^- dt=\int_0^T (\E[Y^n_t]-l_t)nv^n_t dt.
\end{align*}
Therefore, we conclude that 
\begin{align*}
\int_0^T (\E[Y_t]-l_t)v_t dt=0.
\end{align*}
Similarly, $\int_0^T (r_t-\E[Y_t])u_t dt=0$. Now, we only need to check that 
\begin{align*}
K_t=\tilde{K}_t:=\int_0^t v_s ds-\int_0^t u_s ds, \ t\in[0,T]
\end{align*}
The proof is similar with the proof for Eq. (6.30) in \cite{CK}, so we omit it. In fact, the proof for our background is even simpler since the processes related with $K$ are deterministic.
\end{proof}

\begin{remark}
(i) Compared with the condition proposed for the reflecting boundaries in Theorem \ref{main}, we do not need to assume that $\inf_{r\in[0,T]}(r_t-l_t)>0$, that is, the upper obstacle and lower obstacle do not need to be completely separated. 

\noindent (ii) The penalization method in our paper extends the one in \cite{BEH} to the following aspects. First, we could deal with the case of two reflecting barriers. Second, we do not need to assume the benchmark function $l,r$ are constant.
\end{remark}


\begin{thebibliography}{99}
	\bibitem{BCFE} Bally, V., Caballero, M.E., Fernandez, B. and El Karoui, N. (2002) Reflected BSDEs, PDEs and variational inequalities. Preprint inria-00072133.

\bibitem{BEH} Briand, P., Elie, R. and Hu, Y. (2018) BSDEs with mean reflection.  Ann.  Appl. Probab., 28: 482-510.

\bibitem{BH} Briand, P. and Hibon, H. (2021) Particle systems for mean reflected BSDEs. Stochastic Processes and their Applications, 131: 253-275.

\bibitem{BLP} Buckdahn, R., Li, J. and Peng, S. (2009) Mean-field backward stochastic differential equations and related partial differential equations. Stochastic Processes and their Applications, 119: 3133-3154.

\bibitem{BKR} Burdzy, K., Kang, W. and Ramanan, K. (2009) The Skorokhod problem in a time-dependent interval. Stochastic Processes and their Applications, 119: 428-452.

\bibitem{CHM} Chen, Y., Hamad\`{e}ne, S. and Mu, T. (2022) Mean-field doubly reflected backward stochastic differential equations. Numerical Algebra, Control and Optimization, doi:10.3934/naco.2022012.

\bibitem{CM} Cr\'{e}pey, S. and Matoussi, A. (2008) Reflected and doubly reflected BSDEs with jumps. Ann. Probab.,  18(5): 2041-2069.

\bibitem{CK} Cvitanic, J. and Karatzas, I. (1996) Backward stochastic differential equations with reflection and Dynkin games. Ann. Probab., 24(4): 2024-2056.



\bibitem{DEH} Djehiche, B., Elie, R. and Hamad\`{e}ne, S. (2019) Mean-field reflected backward stochastic differential equations, arXiv: 1911.06079.

\bibitem{DQS} Dumitrescu, R., Quenez, M.C. and Sulem, A. (2016) Generalized Dynkin games and doubly reflected BSDEs with jumps. Electronic Journal of Probability, 21: 1-32.

\bibitem{EKPPQ} El Karoui, N., Kapoudjian, C., Pardoux, E., Peng, S. and Quenez, M.C. (1997) {Reflected solutions of backward SDE's, and related obstacle problems for PDE's}. Ann. Probab.,  23(2): 702-737.

\bibitem{EPQ1} El Karoui, N., Pardoux, E. and Quenez, M.C. (1997) {Reflected backward SDE's and  American options}. Numerical Methods in Finance (Cambridge Univ. Press), 1997: 215-231.

\bibitem{EPQ2} El Karoui, N., Peng, S. and Quenez M.C. (2001) A dynamic maximum principle for the optimization of recursive utilities under constraints.  Ann. Appl. Probab., 664-693.


\bibitem{FS21} Falkowski, A. and Slomi\'{n}ski, L. (2021) Mean reflected stochastic differential equations with  two constraints. Stochastic Processes and their Applications, 141: 172-196.

\bibitem{FS} Falkowski, A. and Slomi\'{n}ski, L. (2022) Backward stochastic differential equations with mean reflection and two constraints. Bulletin des Sciences Math\'{e}matiques, 176: 103117.

\bibitem{GIOOQ} Grigorova, M., Imkeller, P., Offen, E., Ouknine, Y. and Quenez, M.C. (2017) Reflected BSDEs when the obstacle is not right-continuous and optimal stopping. Ann. Appl. Probab., 27: 172-196.

\bibitem{GIOQ} Grigorova, M., Imkeller, P., Ouknine, Y. and Quenez, M.C. (2018) Doubly reflected BSDEs and $\mathcal{E}^f$-Dynkin games: beyond the right-continuous case. Electronic Journal of Probability, 23: 1-38.

\bibitem{HL} Hamadene, S. and Lepeltier, J.-P. (2000) Reflected BSDE's and mixed game problem. Stochastic Process. Appl., 85: 177-188.

\bibitem{HLM} Hamadene, S.,  Lepeltier, J.-P. and Matoussi, A. (1997) Double barrier backward SDEs with continuous coefficient. in: El Karoui, L. Mazliak (Eds.), in : Pitman Res. Notes  Math. Ser., 364: 161-171.

\bibitem{HHLLW} Hibon, H., Hu, Y., Lin, Y., Luo, P. and Wang, F. (2018) Quadratic BSDEs with mean reflection. Mathematical Control and Related Fields. 8: 721-738.

\bibitem{K1} Klimsiak, T. (2012) Reflected BSDEs with monotone generator.  Electron. J. Probab., 17(107): 1-25.

\bibitem{K2} Klimsiak, T. (2013) BSDEs with monotone generator and two irregular reflecting barriers. Bull. Sci. Math., 137: 268-321.

\bibitem{KLQT} Kobylanski, M., Lepeltier, J.P., Quenez, M.C. and Torres, S. (2002) Reflected BSDE with superlinear quadratic coefficient. Probability and Mathematical Statistics, 22(1): 51-83.

\bibitem{Li} Li, H. (2023) The Skorokhod problem with two nonlinear constraints, arXiv: 2306.16711.


\bibitem{PX} Peng, S. and Xu, M. (2005) The smallest $g$-supermartingale and reflected BSDE with single and double $L^2$ obstacles. Ann. I. H. Poincare-PR,  41: 605-630.


\bibitem{WY} Wu, Z. and Yu, Z. (2008) Dynamic programming principle for one kind of stochastic recursive optimal control problem and Hamilton-Jacobi-Bellman equation. SIAM Journal on Control and Optimization, 47: 2616-2641.
\end{thebibliography}
 \end{document}